\documentclass[twoside,a4paper,leqno,12pt]{amsproc}
\usepackage[top=30mm,right=30mm,bottom=30mm,left=30mm]{geometry}

\usepackage[pagebackref]{hyperref}
\hypersetup{citecolor=blue, linkcolor=blue, colorlinks=true}
\usepackage{amsmath, amssymb, amsthm, stmaryrd, amsfonts, color, amsrefs}
\usepackage{graphicx, float, array, multicol, rotating}

\newcommand{\A}{\mathcal{A}}

\newcommand{\alt}{\mathrm{A}}

\newcommand{\floor}[1]{{\lfloor #1\rfloor}}
\newcommand{\eps}{\varepsilon}
\newcommand{\PP}{\mathcal{P}}

\newcommand{\sym}{\mathrm{S}}

\renewcommand{\le}{\leqslant}
\renewcommand{\ge}{\geqslant}
\renewcommand{\leq}{\leqslant}
\renewcommand{\geq}{\geqslant}

\usepackage[shortlabels]{enumitem}
\setlist[enumerate]{label=\rm{(\alph*)}}

\theoremstyle{definition}
\newtheorem{definition}{Definition}
\newtheorem{remark}[definition]{Remark}
\theoremstyle{plain}

\newtheorem{theorem}[definition]{Theorem}
\newtheorem{proposition}[definition]{Proposition}
\newtheorem{lemma}[definition]{Lemma}
\newtheorem{corollary}[definition]{Corollary}

\begin{document}

\author{S.\,P. Glasby}
\author{Cheryl E. Praeger}
\author{W.\,R. Unger}
\address{\phantom{|}\kern-1cm S.P. Glasby and Cheryl E. Praeger, Centre for the Mathematics of Symmetry and Computation, University of Western Australia, 35 Stirling Highway, Perth 6009, Australia. \emph{E-mail addresses:} {\tt\texttt{\{stephen.glasby, cheryl.praeger\}@uwa.edu.au}}\vskip3mm\noindent
  W.\,R. Unger, School of Mathematics and Statistics, University of Sydney, NSW 2006, Australia. \emph{E-mail address:} {\tt\texttt{william.unger@sydney.edu.au}}
}

\thanks{Acknowledgements: SPG and CEP gratefully acknowledge support from the Australian Research Council (ARC) Discovery Project DP190100450,
  and WRU gratefully acknowledges the support from ARC project DP160104626.\newline
  2010 Math Subject Classification: 05A05, 68Q17, 11N05.\hfill
 Date: \today.
}

 
\title[Most permutations power to a cycle of small prime length]
      {Most permutations power to \\a cycle of small prime length} 

\date{\today}

\begin{abstract}
  We prove that most permutations of degree $n$ have some power which
  is a cycle of~prime 
  length approximately $\log n$. Explicitly, we show that for
  $n$ sufficiently large, the proportion of such elements is 
  at least  $1-5/\log\log n$ with the prime between  $\log n$
  and $(\log n)^{\log\log n}$. The proportion of even permutations with
  this property is at least $1-7/\log\log n$.
\end{abstract}

\maketitle

\section{Introduction}\label{S:intro}

The symmetric and the alternating groups $\sym_n$ and $\alt_n$ of degree~$n$,
have been viewed as probability spaces with the uniform distribution since the
seminal work of Gruder~\cite{G}, Gon\v{c}arov~\cite{Gon} and
Erd\H{o}s-Tur\'an~\cite{ET2}. There is an analogy between the disjoint
cycle decomposition of a permutation and the prime factorization of an integer
whereby the cycles correspond to prime numbers (see~\cite{aG} or~\cite{Ford}
for a description). The  probability that
a permutation $g\in\sym_n$ is an $n$-cycle is $1/n$, while the probability that
a number $p\le x$ is prime is $1/\log x$, and according to this analogy, $n$ 
corresponds to $\log x$. Counting integers without small prime factors is like counting
permutations without small cycles: both lead to limits reminiscent of 
Mertens' third theorem: 
\[
\lim_{n\to\infty}\left( \log n\prod_{p\le n}\left(1-\frac{1}{p}\right)\right)=e^{-\gamma},\quad \mbox{where $\gamma=0.5772\cdots$} 
\]
is the Euler--Mascheroni constant. Statistical methods have been invaluable for proving
theoretical results. For example, the number of cycles of
a uniformly distributed random element $g\in\sym_n$ (or $g\in\alt_n$)
behaves as $n\to\infty$ like a normal distribution $N(\mu,\sigma^2)$
(\emph{c.f.}~\cite{GS}*{Theorem~1}) with mean and variance
$\mu=\sigma^2=\log n$, see~\cites{Gon,PR}.
This paper focuses on proving the following theorem.
We abbreviate $\log(n)$ and $\log(\log(n))$ as $\log n$ and
$\log\log n$, respectively. All logarithms are to the natural
base $e=2.718\cdots$.

\begin{theorem}\label{T:one}
  Suppose that $G\in\{\alt_n,\sym_n\}$.
  Let $\rho_G$ be the proportion of permutations in $G$ for which
  some power is a cycle of prime length~$p$, where $p$
  lies in the open interval $(\log n,  (\log n)^{\log\log n})$.
  Then, for $n$ sufficiently large,
  $\rho_G \ge 1 - c/\log\log n$ where $c=5$ when $G=\sym_n$, and $c=7$
  when $G=\alt_n$.
\end{theorem}



A permutation $g\in \sym_n$ having a single cycle of length greater than 1
is called a $k$-cycle, where the cycle has length $k$. Their theoretical
importance has long been recognised:  the
presence of a $k$-cycle in a primitive subgroup $G$ of $\sym_n$ was shown to imply 
 that $G$ is $\alt_n$ or $\sym_n$ by Jordan in 1873 in the case where
$k$ is prime and $k\le n-3$ (\cite{Jordan}, 
or see \cite{Wielandt}*{Theorem 13.9}). The same conclusion also holds for 
arbitrary $k < n/2$ by a result in Marggraff's
dissertation~\cite{Marggraff}, see also~\cite{J}*{Corollary~1.3},
\cite{Wielandt}*{Theorem 13.5} and P. M. Neumann's Mathematical
Review MR0424912.

To use Jordan's result for deciding whether a given primitive subgroup $G$
of $\sym_n$ is indeed $\alt_n$ or $\sym_n$, one needs to locate a $p$-cycle
with $p$ a prime, say by choosing random elements from $G$. It is
inefficient to do this directly since the proportion of such elements in
$\alt_n$ or $\sym_n$ is $O(n^{-1})$ (see Section~\ref{ss:recog}). Instead one
searches for a `pre-$p$-cycle', a permutation which powers to a $p$-cycle.
It is shown in~\cite{Seress}*{Lemma~10.2.3} that the proportion of elements
in $\alt_n$ or $\sym_n$ that power to a $p$-cycle
with $n/2<p\le n-3$ is asymptotically $\log 2/\log n$, while an application of the main result 
Theorem~1 of~\cite{BGHP} shows that considering only pre-$p$-cycles 
with $p$ bounded, say $p\leq m$, produces a proportion $c(m)/n^{1/m}$. 
Thus to approach Seress's asymptotic proportion the primes $p$ must be 
allowed to grow unboundedly with $n$.

Quite decisively, and perhaps surprisingly, the third author recently 
showed~\cite{U1}*{Theorem 2} that \emph{almost all} permutations
in $G\in\{\alt_n,\sym_n\}$ are pre-$p$-cycles for some prime~$p$.
The purpose of this paper
is to prove an even stronger statement: namely that, asymptotically, almost
all permutations in $G$ power to a $p$-cycle where the
prime~$p$ is roughly $\log n$ (Theorem~\ref{T:one}) and we
  may derive, from Proposition~\ref{T1} and the proof of Theorem~\ref{T:one},
  explicit values for $n_G$ such that the bounds of Theorem~\ref{T:one}
  hold for all $n\geq n_G$.
On the one hand we  quantify the asymptotic results~\cite{U1}, 
giving a precise analysis with explicit bounds rather than asymptotics. 
In addition we show (which came as a surprise to the authors) that the prime $p$ can be
chosen in a very small interval of length logarithmic in $n$.
Theorem~\ref{T:one} relies on Proposition~\ref{T1} which strengthens a key technical
result~\cite{ET2}*{Theorem VI} of Erd\H{o}s and Tur\'an. Theorem~\ref{T9}
shows that the proportion of pre-$p$-cycles in $G$ is at
least $1-c'\log\log n/\log(n-3)$ for $p$ in the range
$2\le p\le n-3$. We prove in
Remark~\ref{R2} that for all $n\ge5$, the proportion of pre-$p$-cycles
in $\sym_n$ with $2\le p\le n-3$ is at least $1/19$.
In the remainder of this section we comment on uses and proofs of these results.



\subsection{Recognising finite symmetric and alternating groups}\label{ss:recog}

We call a permutation $g\in\sym_n$ a
\emph{pre-$k$-cycle} if it powers to a $k$-cycle where $k>1$. For example,
$g=(1,2,3,4)(5,6)(7,8,9)$ is a pre-3-cycle as $g^4=(7,8,9)$, but $g$
is not a pre-2-cycle. A simple argument (Lemma~\ref{L0}) shows
that the disjoint cycle decomposition of a pre-$k$-cycle contains exactly one
$k$-cycle and all other cycles (if any) have lengths coprime to~$k$.
The number of $k$-cycles in $\sym_n$ equals $n!/c(k)$ where
$c(k)=k(n-k)!$. Hence the proportion $\rho_n$ of cycles in $\sym_n$ equals
$\sum_{k=2}^n 1/c(k) = \sum_{j=0}^{n-2}\frac{1}{(n-j)j!}$. For $n$ even,
\[
 n\rho_n -\kern-1pt \sum_{j=0}^{n-2} \frac1{j!}
=\sum_{j=0}^{n-2}\frac{j}{(n-j)j!} \kern-1pt=\kern-1pt
\sum_{j=0}^{n/2} \frac{j}{(n-j)j!} +\kern-3pt \sum_{j=n/2}^{n-2} \frac{j}{(n-j)j!} \kern-1pt<\kern-1pt
\frac{2}{n}\sum_{j=0}^{n/2} \frac{j}{j!} +\kern-3pt  \sum_{j=n/2}^\infty \frac{j}{j!}
.
\]
Thus $|n\rho_n-e|\to0$ as $n\to\infty$, and so $\rho_n=O(n^{-1})$.
By contrast, the proportion of permutations in $\sym_n$ which are
pre-$k$-cycles (for $2\le k\le n$) approaches 1 as $n\to\infty$. 

Let $X\subseteq \sym_n$ 
and suppose that the subgroup $\langle X\rangle$ of $\sym_n$ generated by $X$ is primitive.
Then a probabilistic algorithm for testing whether
 $\langle X\rangle$  contains $\alt_n$ involves a
random search for a pre-$p$-cycle, where $p$ is a prime less than $n-2$,
and it turns out that such elements have density 1 as $n\to\infty$
(Theorem~\ref{T:one}). If the subgroup
$\langle X\rangle$ does not contain $\alt_n$, we want to limit the (fruitless)
search for pre-$p$-cycles, and for this we need explicit lower bounds on their
density in both $\alt_n$ and $\sym_n$, see Remark~\ref{R1}.

There are two important tools in our proof. The first is to quantify
the sum $\sum_{a<p\le b}1/p^2$ as a function of the real numbers $a$ and $b$.
(Here, and henceforth, $p$ denotes a prime.) The second is to estimate the
number of permutations whose cycle lengths do not have certain
`forbidden lengths'. Erd\H{o}s-Tur\'an~\cite{ET2}*{Theorem~VI} proved that
the probability $\rho$ that the cycle lengths of $g\in\sym_n$ do not lie in
a subset $\A\subseteq\{1,\dots,n\}$ is at most $1/\mu$
where $\mu=\sum_{a\in\A}1/a$. Applying a result of
Ford~\cite{Ford}*{Theorem~2.9} gives the bound $\rho\le e^{1-\mu}$ (take $r=1$,
$T_1=\A$ and $k_1=0$ in Theorem~2.9).
We prove in Proposition~\ref{T1} that $\rho<e^{\gamma-\mu}$,
where $\gamma=0.5772\cdots$. Our bound is less than $2/3$ the size of
bounds of Erd\H{o}s-Tur\'an and Ford. This small improvement is needed
to obtain practically useful bounds in Theorem~\ref{T9}. 
Manstavi\v{c}ius~\cite{M}*{p.\,39} claimed that a slightly weaker bound,
\emph{viz.} $\rho<e^{\gamma-\mu}(1+1/n)$,
followed from a pre-print of his.
Unfortunately, we could not locate his pre-print.

After describing the cycle structure of pre-$k$-cycles
in Section~\ref{S:precycles},
we prove our bound on forbidden cycle lengths in Section~\ref{S:forbidden}.
The sum $\sum_{a<p\le b}1/p^2$ is estimated in Section~\ref{S:estimating},
and the main theorems are proved in Section~\ref{S:Proof}.

\begin{remark}\label{R1}
  Suppose that the proportion of pre-$p$-cycles in $\alt_n$ or $\sym_n$
  is at least $c_0=c_0(n)$.
  Given a primitive permutation group $G\le\sym_n$, then the probability
  that $\alt_n\le G$ and $m$ independent random selections
  from $G$ do \emph{not} find a
  pre-$p$-cycle  for some $p$ with $2\leq p\le n-3$, is at most $(1-c_0)^m$.
  This upper bound is at most a prescribed ``failure probability'' $\eps$
  if and only if  $m\geq \log(\eps)/\log(1-c_0)$. Using Theorem~\ref{T:one}, we can
  take $c_0(n)=1-c/\log\log n$, and a larger value is given in Theorem~\ref{T9}.
  Therefore we may take $c_0$ to be an absolute constant thus requiring
  only  $m=m(\eps)$ random selections. We prove in Remark~\ref{R2} that
  $c_0=0.05$ works for $\sym_n$ and this gives $m(\eps)\ge 20\log(\eps^{-1})$.
  By contrast, the standard analysis of this
  algorithm~\cite{Seress}*{pp.\,226--227} requires
  $m=O(\log(n)\log(\eps^{-1}))$ random selections (because the proportion
  of pre-$p$-cycles in the range $n/2<p<n-2$ is approximately $\log 2/\log n$).
  Thus, our analysis gives an algorithm that is faster by a factor
  of $\log(n)$.
\end{remark}

\section{Precycles}\label{S:precycles}

We begin with some notation. Fix $g\in\sym_n$ and let
$\lambda=\lambda(g)$ be the partition of~$n$
induced by the (disjoint) cycle lengths of $g$. We write $\lambda\vdash n$ and
$\lambda=\langle 1^{m_1}2^{m_2}\cdots n^{m_n}\rangle$ where $m_k$ is
the multiplicity of the part~$k$, so $n=\sum_{k=1}^n km_k$. The vector
$(m_1,\dots,m_n)$ is called the \emph{cycle type} of $g$. A pre-$k$-cycle
$g\in \sym_n$ can be recognized by it \emph{cycle partition} $\lambda(g)$.

\begin{lemma}\label{L0}
  Let $g\in\sym_n$ be a pre-$k$-cycle with
  $\lambda(g)=\langle 1^{m_1}2^{m_2}\cdots n^{m_n}\rangle$.
  Then $m_k=1$ and for $i\ne k$, $m_i>0$ implies $\gcd(k,i)=1$.
  That is, $g$ has a unique $k$-cycle, and its
  other cycles (if any) have lengths  coprime to $k$.
\end{lemma}

\begin{proof}
  Suppose $g$ has disjoint cycle decomposition $g_1\cdots g_r$
  where $g_i$ is a cycle of length $\lambda_i$ and
  $\lambda_1+\cdots+\lambda_r=n$. Note that
  $g^\ell=g_1^\ell\cdots g_r^\ell$ and
  $g_i^\ell$ is a product of $\gcd(\lambda_i,\ell)$ cycles each of
  length $\lambda_i/\gcd(\lambda_i,\ell)$. Suppose $g^\ell$ is a
  $k$-cycle. Then there exists an $i$ for which
  $\lambda_i/\gcd(\lambda_i,\ell)=k$ and $\gcd(\lambda_i,\ell)=1$.
  Also, $\lambda_j/\gcd(\lambda_j,\ell)=1$ for $j\ne i$.
  Thus $\lambda_i=k$, $\gcd(\lambda_i,\ell)=1$. For $j\ne i$,
  $\lambda_j\mid \ell$. Hence $\gcd(k,\lambda_j)=1$.
\end{proof}

\begin{corollary}\label{C0.5}
  For a prime $p$, the set of pre-$p$-cycles in $X\subseteq\sym_n$ equals
  \[
  \left\{g\in X\bigm| m_p(g)
  =1\textup{ and }\sum_{i\ge2}m_{ip}(g)=0\right\}.
  \]
\end{corollary}

A group element whose order is coprime to $p$, for a prime $p$, is
called a \emph{$p'$-element}.
The density of $p'$-elements in
$\sym_{n}$ is $\sigma_n=\prod_{i=1}^{\lfloor n/p\rfloor}(1-1/(ip))$
by~\cite{ET2}*{Lemma~I}. We are grateful to John Dixon for pointing
out to us that the density of pre-$p$-cycles in
$\sym_n$ is $\frac1p\sigma_{n-p}$. The density in $\alt_n$ can be
calculated similarly using~\cite{BLGNPS}*{Theorem~3.3}.

\section{Forbidden cycle lengths}\label{S:forbidden}

Given a set $\mathcal{A}\subseteq\{1,\dots,n\}$ of `forbidden' cycle lengths,
let $\rho$ be the proportion of elements of $\sym_n$ having no cycle length
in $\A$. Erd\H{o}s and Tur\'an prove $\rho\le\mu^{-1}$ in~\cite{ET2}*{Theorem~VI}
where $\mu=\sum_{a\in\A}1/a$. To see that his bound is not optimal,
consider the case when $\A=\{1\}$ and $\mu=1$. In this case,
$\rho$ is the proportion of derangements and the Erd\H{o}s-Tur\'an bound
is unhelpful. However, an inclusion-exclusion argument shows that
$\rho=\sum_{k=0}^n(-1)^k/k!$. As $e^{-1}<\rho$ when $n$ is even and $\rho<e^{-1}$
otherwise, we expect for large $n$ a bound of the form $c/e$ where $c=1+o(1)$.
We prove in Proposition~\ref{T1}
that there is a bound of the form $\rho\le c\kern1pt e^{-\mu}$ where $c=1.781\cdots$.
The constant $c$ equals $e^\gamma$ where $\gamma=0.5772\cdots$ is the
Euler-Mascheroni constant. Our proof uses exponential generating functions,
and has similarities to calculations used by Gruder~\cite{G}, 15 years before Erd\H{o}s and Tur\'an~\cite{ET2}.

If $\mu< 1$, then the Erd\H{o}s-Tur\'an bound gives no information.
On the other hand, if $\mu\ge 1$, then we will show that the upper
bound we obtain in Proposition~\ref{T1}, namely $e^{\gamma-\mu}$, is strictly less
than the bound $\mu^{-1}$ of Erd\H{o}s and Tur\'an. We use the fact
that $e^\mu\geq e \mu$, since the tangent to the curve $y=e^x$ at $x=1$
is $y=ex$. Hence
\[
  e^{\gamma-\mu}=\frac{c}{e^\mu}<\frac{1.782}{e^\mu}\le\frac{1.782}{e\mu}<\frac{2}{3\mu}.
\]
Similarly, our bound improves the bound $e^{1-\mu}$ of
Ford~\cite{Ford}*{Theorem~2.9} by a factor of $e^{1-\gamma}<2/3$. Note that
Ford's bound can be sharpened using estimations in the proof
of Proposition~\ref{T1}.

The centralizer
of $g$ in $\sym_n$ has order $C(\lambda(g)):=\prod_{k=1}^n k^{m_k}m_k!$, and
hence the conjugacy class $g^{\sym_n}$ has size $n!/C(\lambda(g))$. Since the
conjugacy classes in $\sym_n$ are parameterized by the partitions of $n$,
it follows that $\sum_{\lambda\vdash n} n! /C(\lambda)=n!$, or equivalently that
$\sum_{\lambda\vdash n} 1/C(\lambda)=1$.

The reader when trying to compare our analysis with that in~\cite{ET2}
may wish to know that the quantity $L(P)$ defined
on~\cite{ET2}*{p.\,159} for a permutation $P$ should be the number of
cycles of $P$ with length equal
to some $a_\nu\in\A$, and not the number of $a_\nu\in\A$ with the property
that $P$ has a cycle of length $a_\nu$.

\begin{remark}\label{L:AnSn}
  Given a set $X$ and a property $P$, let $\rho(X)$ denote the proportion of
  elements $x\in X$ that have property $P$. If $Y\subseteq X$,
  then $|Y|\rho(Y)\le |X|\rho(X)$. To see this let $P(X)$ be the subset
  of $x\in X$ that have property $P$. Since $|P(X)|=|X|\rho(X)$ and
  $P(Y)\subseteq P(X)$, we have $|Y|\rho(Y)\le|X|\rho(X)$.
\end{remark}





\begin{proposition}\label{T1}
  The proportion
  $\rho(\sym_n)$  of elements of~$\sym_n$ having no cycle length in $\A\subseteq\{1,\dots,n\}$
  satisfies $\rho(\sym_n)<e^{\gamma-\mu}$ where $\mu=\sum_{a\in\A}1/a$
  and $e^\gamma\approx 1.781$. The proportion $\rho(\alt_n)$ of elements
  of~$\alt_n$ with no cycle length in $\A$ satisfies $\rho(\alt_n)<2e^{\gamma-\mu}$.
\end{proposition}

\begin{proof}
  For $k\in\{1,\dots,n\}$, let $p_k$ equal 0 if $k\in\A$, and 1 otherwise.
  Set
  \[
  P(z)=\sum_{k=1}^n\frac{p_kz^k}{k}\qquad\textup{and}\qquad
  Q(z)=e^{P(z)}.
  \]
  As $P(z)$ is a polynomial, $Q(z)$ is a differentiable function of the
  complex variable~$z$, so its Taylor series $\sum_{m=0}^\infty q_mz^m$
  at $z=0$ converges to $Q(z)$ for all $z$. We consider the coefficient~$q_n$.
  Since
  \[
  Q(z)=\prod_{k\not\in\A,\, k\le n} e^{z^k/k}
  =\prod_{k\not\in\A,\, k\le n}\,\sum_{m_k=0}^\infty\frac{(z^k/k)^{m_k}}{m_k!},
  \]
  the term $q_nz^n$ is a sum whose summands have the form
  \[
  \prod_{k\not\in\A,\, k\le n}\frac{z^{km_k}}{k^{m_k}(m_k)!}=\frac{z^n}{C(\lambda)}\qquad
    \textup{for each partition $\lambda=\langle k^{m_k}\rangle_{k\not\in\A}$ with $n=\sum_{k\not\in\A} km_k$.}
  \]
  It follows that $q_n=\sum_{\lambda\vdash n}1/C(\lambda)$ where the parts of
  $\lambda$ do not lie in $\A$ (equivalently $m_k=0$ for all $k\in\A$).
  This is precisely the proportion $\rho(\sym_n)$ of elements of $\sym_n$ none of
  whose cycle lengths lie in $\A$. Alternatively,~\cite{G}*{Eq.~(18)}
  shows that $q_n=\rho(\sym_n)$. 

  We now compute an upper bound for $q_n$. Define the
  $n$th harmonic number to be
  \[
    H_n=\sum_{k=1}^n\frac 1k,\qquad
    \textup{and let}\qquad \gamma=\lim_{n\to\infty} (H_n-\log n)=0.5772\cdots
  \]
  be the Euler-Mascheroni constant. The error term $E(n)=H_n-\log n-\gamma$
  satisfies $0<E(n)<1/(2n)$ by~\cite{H}*{p.\,75}. Since $e^x<1+x+x^2+\cdots=1/(1-x)$
  for $0<x<1$, it follows that $e^{1/(2n)}<1+1/(2n-1)\le 1+1/n$.
  Hence 
  \[
  1<e^{E(n)}<1+\frac{1}{n}\quad\textup{and}\quad
  P(1)=\sum_{k\not\in\A,\,k\le n}\frac 1k=H_n-\mu=\log n +\gamma-\mu+E(n).
  \]
  The previous display implies
  \begin{equation}\label{E:Q}
  Q(1)=e^{P(1)}=e^{\log n}e^{\gamma-\mu}e^{E(n)}
  <n  e^{\gamma-\mu}\left(1+\frac 1n\right)
  =(n+1)e^{\gamma-\mu}.
  \end{equation}
  Differentiating the equation $Q(z)=e^{P(z)}$ gives
  \[
  Q'(z)=P'(z)Q(z),\quad\textup{that is,}\quad \sum_{k=0}^\infty kq_kz^{k-1}=
  \left(\sum_{k=0}^{n-1} p_{n-k}z^{n-1-k}\right)\left(\sum_{k=0}^\infty q_kz^k\right).
  \]
  Equating the coefficients of $z^{n-1}$ on both sides gives
  \[
  nq_n=\sum_{k=0}^{n-1}p_{n-k}q_k\le\sum_{k=0}^{n-1}q_k
  \le \left(\sum_{k=0}^\infty q_k\right)-q_n=Q(1)-q_n.
  \]
  It follows from~\eqref{E:Q} that
  \[
    (n+1)q_n\le Q(1)<(n+1)e^{\gamma-\mu}.
  \]
  Cancelling $n+1$ gives $q_n<e^{\gamma-\mu}$, that is,
  $\rho(\sym_n)<e^{\gamma-\mu}<1.782e^{-\mu}$. Finally, Lemma~\ref{L:AnSn} implies
  that
  $\rho(\alt_n)\le 2\rho(\sym_n)<2e^{\gamma-\mu}<3.563e^{-\mu}$.
\end{proof}

\section{Estimating sums}\label{S:estimating}

In this section we bound the quantity $\mu$ in Proposition~\ref{T1}. One approach is to use
the Stieltjes integral~\cite{RS}*{p.\,67}. Instead, we take an elementary
approach using finite sums, which we now briefly review. In the sequel,
$k$ denotes an integer and $p$ denotes a prime. The
\emph{backward difference} of a function $f(k)$ is defined by
$(\nabla f)(k)=f(k)-f(k-1)$. An easy calculation shows
\begin{align}
  \sum_{k=a+1}^{b} (\nabla f)(k)&=f(b)-f(a),\label{E:Sum1}\quad\textup{and}\\
  \nabla (f(k)g(k)) &= (\nabla f)(k)g(k-1)+f(k)(\nabla g)(k).\label{E:Sum3}
\end{align}
Rearranging~\eqref{E:Sum3} and using~\eqref{E:Sum1} gives
\begin{equation}\label{E:ProductRule2}
  \sum_{k=a+1}^{b} f(k)(\nabla g)(k) = f(b)g(b)-f(a)g(a)
  -\sum_{k=a+1}^{b}(\nabla f)(k)g(k-1).
\end{equation}

Let $\pi(x)=\sum_{p\le x}1$.
Observe that $(\nabla\pi)(k)$ equals 1 if $k$ is prime, and 0 otherwise. Thus
$\sum_{a<p\le b}f(p)=\sum_{a<k\le b}f(k)(\nabla\pi)(k)$ and the latter can
be estimated using~\eqref{E:ProductRule2}.
The following bounds hold for \emph{real} $x\ge17$ by~\cite{RS}*{Theorem~1}.
In order to get a practically useful bounds in Theorem~\ref{T9} we note
that~\eqref{E:PiBounds} holds for all \emph{integers} $x\ge11$:
\begin{equation}\label{E:PiBounds}
  \frac{x}{\log x}\le\pi(x)\le\frac{x}{\log x}\left(1+\frac{3}{2\log x}\right)
  \qquad\textup{for $x\in\{11,12,13,\dots\}$.}
\end{equation}

\begin{lemma}\label{L2}
  Suppose $a$ and $b$ are real numbers with $12\le a\le b$. Then
  \[
  \sum_{a<p\le b}\frac{1}{p^2}\le\frac{2.22}{\lfloor a\rfloor\log \lfloor a\rfloor}-\frac{1.61}{\lfloor b\rfloor\log \lfloor b\rfloor}.
  \]
\end{lemma}

\begin{proof}
  Since the primes in the range $a<p\le b$ coincide with the primes in the range
  $\lfloor a\rfloor<p\le \lfloor b\rfloor$, we henceforth may (and will)
  assume that $a$ and $b$ are \emph{integers} satisfying $12\le a\le b$.
  Also, if $a=b$ the inequality hold, so we assume also that $a+1\le b$.
  Applying~\eqref{E:ProductRule2} and~\eqref{E:PiBounds} gives
  \def\K{\kern-3pt}
\begin{align*}
  \sum_{a<p\le b}\frac{1}{p^2}&=\sum_{k=a+1}^{b} \frac{1}{k^2}(\nabla \pi)(k)=
  \frac{\pi(b)}{b^2}-\frac{\pi(a)}{a^2}
  -\sum_{k=a+1}^{b}\frac{-2k+1}{k^2(k-1)^2}\pi(k-1)\\
  &\K\le \frac{1}{b\log b}\left(1\K+\K\frac{3}{2\log b}\right)\K-\K\frac{1}{a\log a}
  \K+\K\K\sum_{k=a+1}^{b}\K\frac{2k-1}{k^2(k-1)\log(k-1)}\left(1\K+\K\frac{3}{2\log(k-1)}\right)\K.
\end{align*}
Since $12\le a\le k-1$, we have
$1+\frac{3}{2\log(k-1)}\le1+\frac{3}{2\log(12)}<1.61$. As $k-\frac12<k$,
the $\Sigma$-term above is less than
\begin{align*}
\frac{1.61}{\log a}\sum_{k=a+1}^b\frac{2(k-\frac{1}{2})}{k(k-\frac{1}{2})(k-1)}
&= \frac{3.22}{\log a}\sum_{k=a+1}^b\frac{1}{k(k-1)} 
= \frac{3.22}{\log a}\sum_{k=a+1}^b\nabla\left(-\frac{1}{k}\right) \\
&= \frac{3.22}{\log a}\left(\frac{1}{a}-\frac{1}{b}\right) \\
&\le \frac{3.22}{a\log a} - \frac{3.22}{b\log b}.
\end{align*}
Combining the previous two displays gives the desired inequality:
\[
\sum_{a+1<p<b}\frac{1}{p^2} < \left(-1+3.22\right)\frac{1}{a\log a} +
\left(1.61-3.22\right)\frac{1}{b\log b} =
\frac{2.22}{a\log a}-\frac{1.61}{b\log b}.\qedhere
\]
\end{proof}

\begin{lemma}\label{L3}
  Suppose $a,b$ are real numbers with $2\le a\le b$ and $p$ is prime. Then
  \[
  \log\left(\frac{\log b}{\log a}\right)-\frac1{2(\log b)^2}-\frac1{(\log a)^2}
  <  \sum_{a<p\le b}\frac{1}{p}<
  \log\left(\frac{\log b}{\log a}\right)+\frac1{(\log b)^2}+\frac1{2(\log a)^2}.
  \]
\end{lemma}

\begin{proof}
  The proof uses the following bounds~ \cite{RS}*{Theorem~5, Corollary}
\[
  \log\log x+M-\frac{1}{2(\log x)^2}<\sum_{p\le x}\frac{1}{p}< \log\log x+M+\frac{1}{(\log x)^2}
  \qquad\textup{for $x>1$},
\]
where $M=0.26149\cdots$ is the Meissel-Mertens constant. Thus
\begin{align*}
\sum_{a<p\le b} \frac{1}{p}&=\sum_{p\le b} \frac{1}{p}-\sum_{p\le a} \frac{1}{p}\\
&<\left(\log\log b+M+\frac1{(\log b)^2}\right)-\left(\log\log a+M-\frac1{2(\log a)^2}\right)\\
&=\log\left(\frac{\log b}{\log a}\right)+\frac1{(\log b)^2}+\frac1{2(\log a)^2}.
\end{align*}
Similarly
\begin{align*}
\sum_{a<p\le b} \frac{1}{p}&=\sum_{p\le b} \frac{1}{p}-\sum_{p\le a} \frac{1}{p}\\
&>\log\log b+M-\frac1{2(\log b)^2}-\left(\log\log a+M+\frac1{(\log a)^2}\right)\\
&=\log\left(\frac{\log b}{\log a}\right)-\frac1{2(\log b)^2}-\frac1{(\log a)^2}.\qedhere
\end{align*}
\end{proof}


\section{Proof of Theorem~\ref{T:one}}\label{S:Proof}

Suppose $G\in\{\alt_n,\sym_n\}$. Given certain functions $a=a(n)$ and $d=d(n)$,
our strategy is to find a lower bound for the proportion $\rho_G$ of 
pre-$p$-cycles in $G$ with $a(n)<p<a(n)^{d(n)}$. It is shown in~\cite{U1}
that $\rho_G\to 1$ as $n\to\infty$. Theorem~\ref{T:one} was motivated by the needed to analyze
certain Las-Vegas algorithms for permutation groups, we shall adapt a
combinatorial argument in~\cite{U1} to quantify this result.
Henceforth $p$ denotes a prime, and $k$ denotes an integer.

Fix $G\in\{\alt_n,\sym_n\}$. Recall
$\lambda(g)=\langle 1^{m_1(g)}2^{m_2(g)}\cdots n^{m_n(g)}\rangle$
is the partition of $n$ whose parts are the cycle lengths of $g$, and part
$k$ has multiplicity $m_k(g)$. Define
\begin{align*}
  \PP_n&=\{p\mid a(n)<p\le a(n)^{d(n)}\},\\
  T(G)&=\left\{g\in G\mid \textup{$m_p(g)\ge1$ for some $p\in\PP_n$}\right\},\\
  U_p(G)&=\left\{g\in G\bigm| \textup{$m_p(g)\ge1$ and $\sum_{k\ge1}m_{kp}(g)\ge2$}\right\},\textup{ and}\\
  U(G)&=\bigcup_{p\in\PP_n} U_{p}(G).
\end{align*}
Note that $g\in T(G)\setminus U_{p}(G)$ has $m_p(g)\ge1$ and
$\sum_{k\ge1}m_{kp}(g)=1$. Hence $m_p(g)=1$
and $g$ is a pre-$p$-cycle by Lemma~\ref{L0}.
Thus $T(G)\setminus U(G)$ is precisely the set of pre-$p$-cycles in $G$
for some $p\in\PP_n$. In the following proposition, we view $G\in\{\alt_n,\sym_n\}$
as a probability space with the uniform distribution, and we seek a lower
bound for
\begin{align*}
  \rho_G&
  =\textup{Prob}\left(g\in G\textup{ is a pre-$p$-cycle for some $p\in\PP_n$}\right)=\textup{Prob}\left(g\in T(G)\setminus U(G)\right)\\
&=\textup{Prob}\left(g\in G\mid m_p(g)=1\textup{ and }
\sum_{k\ge2}m_{kp}(g)=0\textup{ for some $p\in\PP_n$}\right). 
\end{align*}

\begin{proposition}\label{P9}
  Let $a(n), d(n)$ be functions satisfying $a(n)\ge12$, $d(n)>1$ and
  $a(n)^{d(n)}\le n$ for all $n$. Using the preceding notation and
  $\delta=|\sym_n:G|$,  we have
  \[
  \rho_G=\frac{|T(G)|-|U(G)|}{|G|}\ge1-\frac{2.287\delta}{d(n)}-
  \frac{2.22(\log n-1)}{\lfloor a(n)\rfloor\log \lfloor a(n)\rfloor}-
  \frac{4.4\delta\log n}{a(n)(\log a(n))n}.
  \]
\end{proposition}

\begin{proof}
  Let $\mu=\sum_{p\in\PP_n} \frac{1}{p}$. Write $a$ and $d$ instead
  of $a(n)$ and $d(n)$, and set $b=a^d$. As $a<p\le b\le n$, we have
  $\PP_n\subseteq\{1,\dots,n\}$. Thus Proposition~\ref{T1} gives
\begin{equation}\label{E:Tnbound}
  \frac{|T(G)|}{|G|}=1-\textup{Prob}\left(g\in G\mid
  \textup{$m_{p}(g)=0$ for all $p\in\PP_n$}\right)
  \ge 1-\delta e^{\gamma-\mu}.
\end{equation}

We seek a lower bound for~\eqref{E:Tnbound}. This means
finding a lower bound for $\mu$.
Since $a\ge12$, Lemma~\ref{L3} gives
\begin{equation}\label{E:recip}
  \mu=\sum_{a<p\le a^d} \frac{1}{p}>\log d-\frac1{2(d\log a)^2}-\frac1{(\log a)^2}
  >\log d-0.25.
\end{equation}
Thus it follows from~\eqref{E:Tnbound} and~\eqref{E:recip} that
\begin{equation}\label{E:TG}
\frac{|T(G)|}{|G|} \ge 1-\frac{\delta e^\gamma}{e^\mu}
\ge 1-\frac{\delta e^\gamma}{e^{\log d-0.25}}\ge 1-\frac{2.287\delta}{d}.
\end{equation}

We seek an upper bound for $|U_{p}(G)|$. We (over)count the
number of permutations $g_1g_2g_3\in U_p(G)$ where $g_1,g_2,g_3$ have disjoint supports,
$g_1$ is $p$-cycle and $g_2$ is a $kp$-cycle for some $k\ge1$. 
We may choose $g_1$
in $\binom{n}{p}(p-1)!$ ways, because we take the first element
to be the smallest in the orbit.
Next we choose a $kp$-cycle in $\binom{n-p}{kp}(kp-1)!$ ways. Observe that
$g_1$ is even:
 $p$ is odd as $p>\log 12>2$. Thus when $G=\alt_n$ we must be able to choose
$g_3$ to have the same parity as $g_2$, so that $g_1g_2g_3$ is even.
  In the generic case when $(k+1)p\le n-2$ this is possible.
  The number of choices for $(g_1,g_2,g_3)$ in the generic case is
\[
\binom{n}{p}(p-1)!\binom{n-p}{kp}(kp-1)!\frac{(n-p-kp)!}{\delta}
=\frac{n!}{\delta}\frac{1}{kp^2}.
\]
This is an upper bound for the number of products $g_1g_2g_3$ in the
generic case, and we may halve this upper bound if $k=1$.

  Consider the special case when $n-1\le (k+1)p\le n$ has a solution for
  $k$ and~$p$.
Given $p$, this happens for at most one value of $k$, namely $k=(n-1)/p-1$ or
$k=n/p-1$, where $p$ divides $n$ or $n-1$, respectively. 
We will show that there are very few primes~$p$ for which
$(k+1)p$ lies in $\{n-1,n\}$ for some $k$.  Suppose that $(k+1)p=n-1$ and
  $n-1$ has $r$ (not necessarily distinct) primes divisors greater than $a$.
  Then $n-1>a^r$ and hence there are at most $r<\log_a(n-1)$ choices for~$p$.
  Similarly, if $(k+1)p=n$ there are less than $\log_a(n)$ choices for~$p$.
  Thus the special case has less than $2\log_a(n)$ choices for~$p$.
  Consequently, the contribution in the special case is small, and our
  estimations need not be so careful.

In this special case, arguing as above, the number of choices of
$(g_1, g_2)$ is $n!/(kp^2)$, and the number $n_3$ of choices for $g_3$ is $1$,
unless $G=\alt_n$ and $k$ is even, in which case $n_3=0$. Thus the number
of products $g_1g_2g_3$ is at most $n!/(kp^2)$ and we bound the denominator
as follows:
\[
  kp^2>kap\ge\frac{ka(n-1)}{k+1}\ge\frac{a(n-1)}{2}.
\]
In the special case, the number of choices for $g_1g_2g_3$ is at most
\[
\frac{n!}{kp^2}=\frac{\delta|G|}{kp^2}
\le\frac{2\delta|G|}{a(n-1)}=\eps|G|
\qquad\textup{where $\eps=\frac{2\delta}{a(n-1)}$.}
\]

For each prime $p$, let $\eps_p$ be $\eps$ if $n-1\le (k+1)p\le n$ has a solution for $k$, and 0 otherwise.
  In the generic case we have $(k+1)p\le n-2$ and
  $k\le m:=\lfloor (n-2)/p\rfloor-1$.
The bound $\sum_{k=1}^m 1/k<1+\int_1^m dt/t=1+\log m$ is problematic if $m=0$,
so we replace~$m$ with $m+1$ to get
\[
\frac{|U_{p}(G)|}{|G|}\le\left(\sum_{k=1}^{m+1}\frac{1}{kp^2}\right)
+\eps_p<\frac{1+\log(m+1)}{p^2}+\eps_p.
\]
However,
$1+\log(m+1)\le 1+\log\left(\left\lfloor n/p\right\rfloor\right)<\log n-1$
as $\log p>\log 12>2$. Therefore
\[
\frac{|U_{p}(G)|}{|G|}\le\frac{\log n-1}{p^2}+\eps_p.
\]
As $\eps_p=\eps\ne0$ for
less than $2\log_a(n)$ choices of $p$, we have
\[
\frac{|U(G)|}{|G|}\le\sum_{a<p\le b} \frac{|U_{p}(G)|}{|G|}
\le(\log n-1)\left(\sum_{a<p\le b} \frac{1}{p^2}\right)+2\eps\log_a(n).
\]
Applying the bound for $\sum_{a<p\le b} 1/p^2$
in Lemma~\ref{L2} gives
\[
  \frac{|U(G)|}{|G|}\le (\log n -1)\left(\frac{2.22}{\lfloor a\rfloor\log \lfloor a\rfloor}
  -\frac{1.61}{\lfloor b\rfloor\log \lfloor b\rfloor}\right)+
  \frac{4\delta\log_a(n)}{a(n-1)}.
\]
Since $12\le a<n$ and $4/(n-1)<4.4/n$ holds for $n\ge13$, we have
\begin{equation}\label{E:Un}
  \frac{|U(G)|}{|G|}\le (\log n -1)\left(\frac{2.22}{\lfloor a\rfloor\log \lfloor a\rfloor}
  -\frac{1.61}{\lfloor b\rfloor\log \lfloor b\rfloor}\right)+
  \frac{4.4\delta\log_a(n)}{an}. 
\end{equation}

Now~\eqref{E:TG} and~\eqref{E:Un} give
\begin{equation}\label{E:lb}
  \rho_G=\frac{|T(G)|-|U(G)|}{|G|}\ge 1-\frac{2.287\delta}{d}-
  \frac{2.22(\log n -1)}{\lfloor a\rfloor\log \lfloor a\rfloor}
  -\frac{4.4\delta\log n}{a(\log a)n}.\qedhere
\end{equation}
\end{proof}

\begin{proof}[Proof of Theorem~\ref{T:one}]
  Set $a=\log n$ and $d=\log\log n$. Suppose that $n\ge e^{12}$.
  Then
  $a\ge12$, and also $a^d=(\log n)^{\log\log n}<n$.
  Using Proposition~\ref{P9} and the
  inequalities  $a-1<\lfloor a\rfloor$ and $4.4/\log a<2$ gives
\begin{align*}
\rho(G) &\ge
1 - \frac{2.287\delta}{d} - \frac{2.22(a-1)}{\floor{a}\log\floor{a}}
-\frac{4.4\delta a}{a(\log a)n}\\
&>
1 - \frac{2.287\delta}{d} - \frac{2.22}{\log(a-1)}-\frac{2\delta}{n}.
\end{align*}
However,
\[
\log(a-1)=\log a+\log\left(1-\frac 1a\right)\ge\log a
+\log\left(\frac{11}{12}\right)\ge c'\log a,
\]
where $c'=1-\log(12/11)/\log(12)=\log(11)/\log(12)$
  and $2/n<10^{-3}/\log\log n$, so
\[
  \rho_G>
  1-\frac{2.287\delta}{\log\log n}-\frac{2.22}{c'\log\log n}
    -\frac{0.001\delta}{\log\log n}
  \ge1-\frac{2.288\delta}{\log\log n}-\frac{2.301}{\log\log n}.
  \]
  This is at least $1-c/\log\log n$ where $c=4.6$ if $\delta=1$ and $c=6.9$
  if $\delta=2$.
\end{proof}

Allowing pre-$p$-cycles with larger $p$ gives us a sharper lower bound for $\rho_G$.

\begin{theorem}\label{T9}
Suppose that $n \ge e^{12}$ and $A_n\le G\le S_n$. Let $\rho_G$ be the
proportion of permutations in $G$ that power to a cycle with prime
length $p\le n-3$. Then
\[
  \rho_G\ge 1 - \frac{(4.58\delta+0.17)\log\log n}{\log (n-3)}\qquad\textup{where $\delta=\frac{n!}{|G|}\in\{1,2\}$.}
\]
\end{theorem}
\begin{proof}
  Set $a(n) = (\log n)^2$ in Proposition~\ref{P9} and suppose
  that $a(n)^{d(n)}=n-3$.
  (The hypotheses $a(n)\ge12$, $d(n)>1$ and $a^d\le n$ of Proposition~\ref{P9}
  clearly hold.)
  Then $d(n)=\log (n-3)/\log a= \log (n-3)/2\log\log n$ and Proposition~\ref{P9} gives
  \begin{align*}
    \rho_G&\ge 1 - \frac{2.287\delta}{d(n)}
         -\frac{2.22(\log n-1)}{\lfloor(\log n)^2\rfloor\log\lfloor(\log n)^2\rfloor}-
         \frac{4.4\delta\log n}{a(\log a)n}\\
    &>1 - \frac{4.574\delta\log\log n}{\log(n-3)}-
\frac{2.22(\log n-1)}{\lfloor(\log n)^2\rfloor\log\lfloor(\log n)^2\rfloor}-
\frac{2.2\delta}{\log (n-3)(\log\log n)n}.
  \end{align*}
  Note that
  $\lfloor(\log n)^2\rfloor> (\log n)^2-1=(\log n-1)(\log n+1)$ and
  since $n\ge e^{12}$ we have
\[
\frac{2.22(\log n-1)}{\lfloor(\log n)^2\rfloor}<\frac{2.22}{\log n +1}
<\frac{2.05}{\log(n-3)}
\qquad\textup{and}\qquad
\frac{2.2\delta}{(\log\log n)n}\le \frac{\delta\log\log n}{10^3}.
\]
Using these inequalities, and combining the $\delta$ terms, shows that
\[
\rho_G\ge 1 - \frac{4.575\delta\log\log n}{\log (n-3)}
- \frac{2.05}{\log(n-3)\log(12^2)}
\ge 1 - \frac{(4.58\delta + 0.17)\log\log n}{\log (n-3)}.\qedhere
\]
\end{proof}


\begin{remark}\label{R2}
  Suppose that $n\ge5$. We prove that the proportion $\pi_n$ of
  elements of $\sym_n$ that are pre-$p$-cycles for some~$p$
  with $2\le p\le n-3$ is at least $1/19$. We know that
  $\pi_n\ge\pi_0$ where $\pi_0:=\sum_{n/2<p\le n-3}1/p$. A simple computation
  with {\sc Magma}~\cite{Magma} shows that $\pi_0\ge 1/19$ for all $n$
  satisfying $5\le n\le 400{,}000$.
  For $n>400{,}000>e^{12}$ we have $\pi_n\ge 1-4.75\log\log n/\log(n-3)>1/19$
  by Theorem~\ref{T9}. Precise computations of $\pi_n$ for $n \le 50$ suggest
  that $\pi_n>1/3$ may even hold for all $n\ge5$.
\end{remark}

\vskip2mm{\sc Acknowledgement.} 
We thank the referee for some very helpful suggestions.

\end{document}